\documentclass[reqno]{amsart}
\usepackage{amssymb, a4wide, mathdots, url,hyperref,graphicx}

\usepackage{amsfonts}
\usepackage{amssymb}
\usepackage{amsfonts, amsmath, amsthm, amssymb}
\usepackage{graphicx}
\usepackage{listings}
\usepackage{xcolor}
\usepackage{thmtools}
\usepackage{mathrsfs}
\usepackage{hyperref}
\usepackage[all]{xy}
\usepackage{dsfont}
\usepackage{accents}
\usepackage{textcomp}
\usepackage{enumerate}
\usepackage[utf8]{inputenc}

\newcommand{\diag}{\operatorname{diag}}

\theoremstyle{plain}
\newtheorem{thm}{Theorem}[section]

\newtheorem{prop}[thm]{Proposition}
\newtheorem{lem}[thm]{Lemma}
\newtheorem{cor}[thm]{Corollary}

\newtheorem{defn}[thm]{Definition}
\newtheorem{defn/thm}[thm]{Definition/Theorem}
\newtheorem{defn/prop}[thm]{Definition/Proposition}

\theoremstyle{remark}

\begin{document}
	
	\title{On irreducibility of certain low dimensional automorphic Galois representations}
	
	\author{Boyi Dai}
	\address{Department of Mathematics, HKU, Pokfulam, Hong Kong}
	\email{Daiboy@connect.hku.hk}

	\begin{abstract}
		We study irreducibility of Galois representations $\rho_{\pi,\lambda}$ associated to a $n=7$ or 8-dimensional regular algebraic essentially self-dual cuspidal automorphic representation $\pi$ of $\text{GL}_n(\mathbb{A}_\mathbb{Q})$. We show $\rho_{\pi,\lambda}$ is irreducible for all but finitely many $\lambda$ under the following extra conditions.
		\begin{enumerate}[(i)]
			\item If $n=7$, and there exists no $\lambda$ such that the Lie type of $\rho_{\pi,\lambda}$ is the standard representation of exceptional group $\textbf{G}_2$.
			\item If $n=8$, and when there exist infinitely many $\lambda$ such that the Lie type of $\rho_{\pi,\lambda}$ is the spin representation of $\text{SO}_7$, we assume there exist no three distinct Hodge-Tate weights form a 3-term arithmetic progression.
		\end{enumerate}
	\end{abstract}
	
	\maketitle


	\section{Introduction}
	
	It is a folklore conjecture (see \cite{Ra08}) that the Galois representations associated to algebraic cuspidal automorphic representations of $\text{GL}_n(\mathbb{A}_F)$ of a number field $F$ are irreducible. For classcial modular forms this was proved in \cite{Ri77}, and the proof was extended to Hilbert modular forms in \cite{Ta95}. For $n=3$, $F$ is CM and $\pi$ is essentially self-dual, the result was proved in \cite[Theorem 2.2.1]{BR92}. For $n=3$, $F$ is totally real and without essentially self-dual condition, the result was proved in \cite{BH25}. For $n=4$, $F$ being totally real and $\pi$ is essentially self-dual, the irreducibility for almost all $\ell$ was proved in \cite{Ra13}.
	
	For general dimension, due to the work of many people, one can attach a strictly compatible system (see \autoref{Aluminium}) $\{\rho_{\pi,\lambda}\}$ to an algebraic, regular, cuspidal, essentially self-dual automorphic representation $\pi$ of $\text{GL}_n(\mathbb{A}_F)$ where $F$ is a CM or a totally real field (see \autoref{Silicon}). The irreducibility for almost all $\lambda$ when $n\le6$ was proved in \cite{Hu23b}. For general $n$, the irreducibility for a positive density set of $\lambda$ was proved in \cite{PT15}. When $4\nmid n$ and $7\nmid n$, the irreducibility for a density one set of $\lambda$ was proved in \cite{FW25}. When $F$ is totally real and some irreducible $\rho_{\pi,\lambda_0}$ is of certain $A_1$ type, the irreducibility of all $\rho_{\pi,\lambda}$ was proved in \cite{HL24}, \cite{HL25}. 
	
	The present paper continues to investigate the irreducibility of low dimensional automorphic Galois representations. We focus on $n=7$ and 8. For a $\lambda$-adic semisimple Galois representation $\rho:\text{Gal}_{\mathbb{Q}}\to\text{GL}_n(\overline{E}_{\lambda})$, the Zariski closure $\textbf{G}$ of its image inside $\text{GL}_{n,\overline{E}_{\lambda}}$ is a reductive group. Denote by $\textbf{G}^{\text{der}}=[\textbf{G}^{\circ},\textbf{G}^{\circ}]$ its derived subgroup, which is semisimple. Our main result is:
	\begin{thm}\label{main}
		Let $\{\rho_{\pi,\lambda}:\text{Gal}_{\mathbb{Q}}\to\text{GL}_n(\overline{E}_{\lambda})\}_{\lambda}$ be the $E$-rational strictly compatible system of $\mathbb{Q}$ associated to a regular algebraic essentially self-dual cuspidal automprhic representation $\pi$ of $\text{GL}_n(\mathbb{A}_{\mathbb{Q}})$ where $n=7$ or 8. Moreover we require:
		\begin{enumerate}[(i)]
			\item If $n=7$, there exists no $\lambda$ such that tautological representation of $\textbf{G}_{\lambda}^{\text{der}}$ is the standard representation of exceptional group $\textbf{G}_2$.
			\item If $n=8$, and when there exist infinitely many $\lambda$ such that tautological representation of $\textbf{G}_{\lambda}^{\text{der}}$ is the spin representation of $\text{SO}_7$, we assume for any three distinct Hodge-Tate weights $\{a,b,c\}$ one has $a+b\not=2c$.
		\end{enumerate}
		Then $\rho_{\pi,\lambda}$ is irreducible for all but finitely many $\lambda$.
	\end{thm}
	
	We organize the paper as following. In section 2 we give preliminaries including $\ell$-independence properties of compatible systems, potential automorphy theorem, big image results and some $p$-adic Hodge theoretic lift results. In section 3 we follow the treatment in \cite{Xi19} to assume there exist infinitely many $\lambda$ such that $\rho_{\pi,\lambda}$ is Lie-irreducible. Finally in section 4 we prove the main theorem.
	
	\section{Preliminaries}
	\subsection{Compatible systems and $\ell$-independence}
	
	\begin{defn}
		Let $K$ be a number field. An $n$-dimensional $E$-rational Serre compatible system of $\mathrm{Gal}_K$ is the datum
		$$\mathcal{M}=\left(E,S,\{p_v(T)\},\{\rho_{\lambda}\}\right)$$
		where:
		\begin{itemize}
			\item $E$ is a number field.
			\item $S$ is a finite set of primes of $K$ called exceptional set.
			\item $P_v(T)\in E[T]$ is a degree $n$ monic polynomial for each prime $v\not\in S$ of $K$.
			\item $\rho_{\lambda}:\mathrm{Gal}_K\to\mathrm{GL}_n(E_{\lambda})$ is an $n$-dimensional continuous Galois representation.
		\end{itemize}
		such that:
		\begin{enumerate}[(i)]
			\item $\rho_{\lambda}$ is unramified outside $S\cup S_{\lambda}$, where $S_{\lambda}$ is the primes of $K$ that divide the same rational prime as $\lambda$.
			\item For each $v\not\in S\cup S_{\lambda}$, the characteristic polynomial of $\rho_{\lambda}(\text{Frob}_v)$ is $P_v(T)$. 
		\end{enumerate}
		
		A Serre compatible system is called semisimple if each $\rho_{\lambda}$ is semisimple.
	\end{defn}
	
	Consider a semisimple $\lambda$-adic Galois representation $\rho:\mathrm{Gal}_K\to\text{GL}_n(E_{\lambda})$. Denote by $\textbf{G}$ the Zariski closure of the image inside algebraic group $\text{GL}_{n,E_{\lambda}}$, which is called the algebraic monodromy group of $\rho$. As $\rho$ is semisimple, the identity component $\textbf{G}^{\circ}$ is a reductive group. We write $\textbf{G}^{\text{der}}=[\textbf{G}^{\circ},\textbf{G}^{\circ}]$ to be the derived subgroup of $\textbf{G}^{\circ}$, which is semisimple. To describe $\ell$-independence properties of compatible systems, we need the following concepts:
	
	\begin{defn}
		Let $F$ be a field and $\textbf{G}\subseteq\mathrm{GL}_{n,F}$ be a reductive subgroup. Denote by $\overline{F}$ a fixed algebraic closure of $F$.
		\begin{enumerate}[(i)]
			\item Denote by $\textbf{T}$ a maximal torus of $\textbf{G}\times\overline{F}$ and by $\textbf{T}'$ a maximal torus of $\textbf{G}^{\mathrm{der}}\times\overline{F}$. Then the formal character (resp. formal bi-character) of $\textbf{G}$ is the conjugacy class of $\textbf{T}$ in $\mathrm{GL}_{n,\overline{F}}$ (resp. conjugacy class of the chain $\textbf{T}'\subseteq\textbf{T}$ in $\mathrm{GL}_{n,\overline{F}}$.).
			\item Given two fields $F_1,F_2$ and two reductive groups $\textbf{G}_i\subseteq\mathrm{GL}_{n_i,F_i}, i=1,2$. We say they have same formal character (resp. formal bi-character), if $n_1=n_2=n$ and there exists a split $\mathbb{Z}$-subtorus $\textbf{T}_{\mathbb{Z}}\subseteq\mathrm{GL}_{n,\mathbb{Z}}$ (resp. a chain of split $\mathbb{Z}$-subtori $\textbf{T}_{\mathbb{Z}}'\subseteq\textbf{T}_{\mathbb{Z}}\subseteq\mathrm{GL}_{n,\mathbb{Z}}$) such that $\textbf{T}_{\mathbb{Z}}\times\overline{F_i}$ (resp. $\textbf{T}_{\mathbb{Z}}'\times\overline{F}_i\subseteq\textbf{T}_{\mathbb{Z}}\times\overline{F}_i$) is contained in formal character (resp. formal bi-character) of $\textbf{G}_i$ for each $i$. This defines an equivalence relation on formal characters (resp. formal bi-characters) of reductive groups over different fields.
			\item Let $\{F_i\}$ be a family of fields and $\{\textbf{G}_i\subseteq\mathrm{GL}_{n,F_i}\}$ be a family of reductive groups. We say they have same formal character (resp. same formal bi-character) if they belong to the same class under the equivalence relation in (ii). We say they have bounded formal characters (resp. bounded formal bi-characters) if they belong to finitely many classes under the equivalence relation in (ii).
		\end{enumerate}
	\end{defn}
	
	The following results are standard $\lambda$-independence properties on algebraic monodromy groups.
	\begin{thm}\cite{Se81}, \cite{Se84}, \cite[Theorem 3.19]{Hu13}\label{Hydrogen}.
		Given an $E$-rational semisimple Serre compatible system $\{\rho_{\lambda}:\mathrm{Gal}_K\to\text{GL}_n(E_{\lambda})\}$. Denote by $\textbf{G}_{\lambda}$ the algebraic monodromy group of $\rho_{\lambda}\otimes\overline{E_{\lambda}}$
		\begin{enumerate}[(i)]
			\item The component group $\pi_0(\textbf{G}_{\lambda})=\textbf{G}_{\lambda}/\textbf{G}_{\lambda}^{\circ}$ is independent of $\lambda$. In particular the connectedness of $\textbf{G}_{\lambda}$ is independent of $\lambda$ and one has a smallest extension $K'/K$ such that when restricting the compatible system to $\mathrm{Gal}_{K'}$, each algebraic monodromy group is connected.
			\item The formal bi-character of the tautological representation $\textbf{G}_{\lambda}\hookrightarrow\text{GL}_{n,\overline{E_{\lambda}}}$ and hence the rank and semisimple rank of $\textbf{G}_{\lambda}$ are independent of $\lambda$.
		\end{enumerate}
	\end{thm}
	
	The following definition of compatible systems has extra conditions than Serre compatible systems which makes them more treatable.
	\begin{defn}\label{Aluminium}
		Let $K$ be a number field. An $n$-dimensional $E$-rational strictly compatible system of $\mathrm{Gal}_K$ is the datum
		$$\mathcal{M}=\left(E,S,\{P_v(T)\},\{\rho_{\lambda}\},\{\text{HT}_{\tau}\},\{\text{WD}_v\}\right)$$
		
		where:
		\begin{itemize}
			\item $E$ is a number field.
			\item $S$ is a finite set of primes of $K$ called exceptional set.
			\item $P_v(T)\in E[T]$ is a degree $n$ monic polynomial for each prime $v\not\in S$ of $K$.
			\item $\rho_{\lambda}:\mathrm{Gal}_K\to\text{GL}_n(\overline{E_{\lambda}})$ is an $n$-dimensional continuous semisimple Galois representation.
			\item $\text{HT}_{\tau}$ is a multiset of $n$ integers for each embedding $\tau: K\hookrightarrow\overline{E}$.
			\item $\text{WD}_v$ is a semisimple Weil-Deligne representation of $K_v$ for each prime $v$.
		\end{itemize}
		
		such that:
		\begin{enumerate}[(i)]
			\item Each $\rho_{\lambda}$ is a geometric representation in the sense of Fontaine-Mazur with exceptional set $S$, this means
			\begin{itemize}
				\item $\rho_{\lambda}$ is unramified outside $S\cup S_{\lambda}$, where $S_{\lambda}$ is the primes of $K$ that divide the same rational prime as $\lambda$.
				\item If $v\in S_{\lambda}$ then $\rho_{\lambda}|_{\mathrm{Gal}_{K_v}}$ is de Rham.
			\end{itemize}
			Moreover, $\rho_{\lambda}|_{\mathrm{Gal}_{K_v}}$ is crystalline if $v\in S_{\lambda}$ and $v\not\in S$.
			\item For each $v\not\in S\cup S_{\lambda}$, the characteristic polynomial of $\rho_{\lambda}(\text{Frob}_v)$ is $P_v(T)$. 
			\item For each embedding $\tau:K\hookrightarrow\overline{E}$ and each $E$-embedding $\overline{E}\hookrightarrow\overline{E_{\lambda}}$, the Hodge-Tate weights of $\rho_{\lambda}$ is $\text{HT}_{\tau}$.
			\item For each $v\not\in S_{\lambda}$ and each isomorphism $\iota:\overline{E_{\lambda}}\cong\mathbb{C}$, the Frobenius semisimplified Weil-Deligne representation $\iota\text{WD}(\rho_{\lambda}|_{\mathrm{Gal}_{K_v}})^{F-\text{ss}}$ is isomorphic to $\text{WD}_v$.
		\end{enumerate}
	\end{defn}
	
	A Hodge-Tate representation is called regular, if its Hodge-Tate weights are distinct. Under this condition, the following result shows that one can descend the coefficients of a strictly compatible system to $E_{\lambda}$ after enlarging $E$, which makes it a Serre compatible system.
	\begin{lem}\cite[Lemma 5.3.1.(3)]{BLGGT14}\label{Phosphorus}
		Let $\{\rho_{\lambda}\}$ be an $E$-rational strictly compatible system of $K$. Suppose $\mathcal{M}$ is regular, then after replacing $E$ with a finite extension, we may assume that for any open subgroup $H$ of $\mathrm{Gal}_K$, any $\lambda$ and any $H$-subrepresentation $\sigma$ of $\rho_{\lambda}$, the representation $\sigma$ is defined over $E_{\lambda}$.
	\end{lem}
	\subsection{Rectangular representations and $\ell$-independence}
	
	Let $\rho:\mathfrak{g}\to\mathrm{End}(V)$ be a finite dimensional representation of a complex Lie algebra $\mathfrak{g}$. Denote by $\Lambda$ a weight lattice (with respect to some fixed Cartan subalgebra $\mathfrak{t}$), by $\Xi$ the (multi)set of weights. For non-negaive integer $d$, denote by $Z_d=\{-d,-d+2,-d+4,\cdots,d-2,d\}$. Let $r$ be the rank of $\mathfrak{g}$.
	\begin{defn}\cite[Section 1.1]{HL25}\
		\begin{enumerate}[(i)]
			\item  $\rho$ is called rectangular if every weight in $\Xi$ is of multiplicity one and there exist an $\mathbb{R}$-isomorphism $\iota:\Lambda\otimes\mathbb{R}\to\mathbb{R}^n$ and non-negative integers $d_1,d_2,\cdots,d_r$ such that
			$$\iota(\Xi)=Z_{d_1}\times Z_{d_2}\times\cdots\times Z_{d_r}$$
			The (multi)set $\{d_i+1,1\le i\le r\}$ is called the set of lengths of $\rho$. The representation $\rho$ is called hypercubic if $d_1=d_2=\cdots=d_r$. A rectangular representation is called indecomposable if it is not equivalent to an external tensor product of two rectangular representations.
			\item Let $\widetilde{\rho}:\mathrm{Gal}_K\to\mathrm{GL}_n(E_{\lambda})$ be a semisimple $\lambda$-adic Galois representation of a number field $K$. Denote by $\textbf{G}$ the algebraic monodromy group. Fix some embedding $E_{\lambda}\hookrightarrow\mathbb{C}$ and consider the complex base change $\textbf{G}_{\mathbb{C}}\to\mathrm{GL}_n(\mathbb{C})$. Consider the associated complex Lie algebra representation $\mathrm{Lie}(\textbf{G}_{\mathbb{C}})\to\mathrm{GL}_n(\mathbb{C})$ and denote by $\rho$ its restriction to the semisimple part $\mathrm{Lie}(\textbf{G})^{\text{ss}}$. We call $\widetilde{\rho}$ rectangular or hypercubic or indecomposable if $\rho$ is so. 
		\end{enumerate}
	\end{defn}
	
	One of the main results in \cite{HL25} gave a complete classification of rectangular representations of complex Lie algebras.
	\begin{thm}\cite[Theorem 1.1]{HL25}
		Let $(\mathfrak{g},\rho)$ be a faithful rectangular Lie algebra representation of a complex semisimple Lie algebra $\mathfrak{g}$. Fix a decomposition $\mathfrak{g}=\mathfrak{g}_1\times\mathfrak{g}_2\times\cdots\times\mathfrak{g}_k$ where $\mathfrak{g}_1$ denotes the product of $A_1$-factors and $\mathfrak{g}_2,\cdots,\mathfrak{g}_k$ denote other simple factors. Then the following assertions hold.
		\begin{enumerate}[(i)]
			\item There exist faithful rectangular representation $(\mathfrak{g}_1,\rho_1)$ and faithful indecomposable hypercubic representations $(\mathfrak{g}_i,\rho_i)$ for $2\le i\le k$ such that
			$$(\mathfrak{g},\rho)=\bigotimes_{i=1}^k(\mathfrak{g}_i,\rho_i)$$
			\item The rectangular representation $(\mathfrak{g}_1,\rho_1)$ admits an external tensor product of indecomposable hypercubic representations
			$$(\mathfrak{g}_1,\rho_1)=\bigotimes_{j=1}^s(\mathfrak{g}_{1,j},\rho_{1,j})$$
			such that $\mathfrak{g}_1=\prod_{j=1}^s\mathfrak{g}_{1,j}$ and each $\rho_{1,j}$ is one of the following.
			\begin{enumerate}[(a)]
				\item $(A_1,\mathrm{Sym}^r(\mathrm{Std})),r\in\mathbb{N}$.
				\item $(A_1,\mathrm{Sym}^{r_1}(\mathrm{Std}))\otimes\mathrm{Sym}^{r_2}(\mathrm{Std})),r_1,r_2\in\mathbb{Z}_{\ge0},|r_1-r_2|=1$.
				\item $(A_1\times A_1,(\mathrm{Std}\otimes\mathds{1})\oplus(\mathds{1}\otimes\mathrm{Std}))=(D_2,\mathrm{Spin})$.
			\end{enumerate}
			\item The hypercubic representation $(\mathfrak{g}_i,\rho_i)$ for $2\le i\le k$ is one of the following.
			\begin{enumerate}[(a)]
				\item $(B_2,\mathrm{Std}\oplus\mathrm{Spin})$.
				\item $(B_m,\mathrm{Spin}),m\ge2$.
				\item $(A_3,\mathrm{Std}\oplus\mathrm{Std}^{\vee})$.
				\item $(D_m,\mathrm{Spin}),m\ge4$.
			\end{enumerate}
			\item The external tensor products in (i) and (ii) are unique up to permutations of the $A_1$-factors and the non-$A_1$ factors.
		\end{enumerate}
	\end{thm}
	
	Due to \autoref{Hydrogen}.(ii), for a semisimple Serre compatible system $\{\rho_{\lambda}\}$, if one $\rho_{\lambda_0}$ is rectangular, then all $\rho_{\lambda}$ are rectangular with same (multi)set of lengths. Hence it makes sense to call a compatible system rectangular and define its (multi)set of lengths. In the proof we use the following direct consequence.
	\begin{prop}\label{Chromium}
		Let $\{\rho_{\lambda}\}$ be an 8-dimensional semisimple rectangular Serre compatible system such that some $\rho_{\lambda_0}$ is Lie-irreducible. Denote by $\textbf{G}_{\lambda}^{\mathrm{der}}$ the derived subgroup of algebraic monodromy group of $\rho_{\lambda}$. Denote by $\mathcal{L}$ the (multi)set of lengths. Then exactly one of the following happens.
		\begin{enumerate}[(i)]
			\item $\mathcal{L}=\{8\}$ and all $\textbf{G}_{\lambda}^{\mathrm{der}}$ equal to
			$$(\mathrm{SL}_2,\mathrm{Sym}^7(\mathrm{Std}))$$
			\item $\mathcal{L}=\{2,4\}$ and all $\textbf{G}_{\lambda}^{\mathrm{der}}$ equal to
			$$(\mathrm{SL}_2,\mathrm{Std})\otimes(\mathrm{SL}_2,\mathrm{Sym}^3(\mathrm{Std}))$$
			\item $\mathcal{L}=\{2,2,2\}$ and $\textbf{G}_{\lambda}^{\mathrm{der}}$ equals to one of the following.
			\begin{enumerate}[(a)]
				\item $(\mathrm{SL}_2,\mathrm{Std})\otimes(\mathrm{SL}_2,\mathrm{Std})\otimes(\mathrm{SL}_2,\mathrm{Std})$.
				\item $(\mathrm{SL}_2,\mathrm{Std})\otimes(\mathrm{SL}_2\times\mathrm{SL}_2,(\mathrm{Std}\otimes\mathds{1})\oplus(\mathds{1}\otimes\mathrm{Std}))$.
				\item $(\mathrm{SL}_4,\mathrm{Std}\oplus\mathrm{Std}^{\vee})$.
				\item $(\mathrm{SL}_2,\mathrm{Std})\otimes(\mathrm{Sp}_4,\mathrm{Std})$.
				\item $(\mathrm{SO}_7,\mathrm{Spin})$.
			\end{enumerate}
		\end{enumerate}
	\end{prop}
	
	\subsection{Automorphic Galois representations}
	
	\begin{defn/prop}\cite[Section 2.1]{BLGGT14}\label{Scandium}
		Let $F$ be a CM or a totally real field. Denote by $F^+$ the maximal totally real subfield of $F$. Let $E$ be a number field and $\lambda$ be a prime of $E$.
		\begin{enumerate}[(i)]
			\item A $\lambda$-adic Galois representation $\rho:\text{Gal}_F\to\text{GL}_n(\overline{E}_{\lambda})$ is called essentially self-dual, if there exists a character $\chi:\text{Gal}_{F^+}\to\overline{E}_{\lambda}^*$ such that for some (hence all) infinite place $v$ of $F^+$, there exists $\varepsilon_v\in\{\pm1\}$ and a non-degenerate pairing $\langle-,-\rangle_v$ on $\overline{E}_{\lambda}^n$ such that
			$$\langle x,y\rangle_v=\varepsilon_v\langle y,x\rangle_v$$
			$$\langle \rho(g)x,\rho(c_vgc_v)y\rangle_v=\chi(g)\langle x,y\rangle_v$$
			for all $x,y\in\overline{E}_{\lambda}^n$ and all $g\in\text{Gal}_F$. Here $c_v$ is the complex conjugation associated to $v$. For $F$ being a CM field we further assume $\varepsilon_v=-\chi(c_v)$.
			\item Moreover, $\rho$ is called totally odd if $\varepsilon_v=1$ for all infinite place $v$.
			\item If $F$ is totally real, then $\rho$ is essentially self-dual, if and only if it either factors through $\text{GSp}_n(\overline{E}_{\lambda})$ (if $\chi(c_v)=-\varepsilon_v$) or $\text{GO}_n(\overline{E}_{\lambda})$ (if $\chi(c_v)=\varepsilon_v$). In particular there exists some continuous character $\chi:\text{Gal}_K\to\overline{E}_{\lambda}^*$ called similitude character, such that $\rho\cong\rho^{\vee}\otimes\chi$. In such case, $\rho$ is totally odd if and only if it either factors through $\text{GSp}_n$ with totally odd similitude character (i.e. for any complex conjugation $c$ one has $\chi(c)=-1$) or it factors through $\text{GO}_n$ with totally even similitude character (i.e. for any complex conjugation $c$ one has $\chi(c)=1$).
		\end{enumerate}
	\end{defn/prop}
	
	Automorphic Galois representations refer to the ones arising from following result.
	\begin{thm}\cite[Theorem 2.1.1]{BLGGT14}\label{Silicon}
		Let $F$ be a CM or a totally real field. Suppose that $(\pi,\chi)$ is a regular algebraic cuspidal polarized automorphic representation of $\text{GL}_n(\mathbb{A}_F)$. Then there exists a CM field $E$ and an $E$-rational Serre compatible system
		$$\{\rho_{\pi,\lambda}:\mathrm{Gal}_F\to\mathrm{GL}_{n}(\overline{E}_{\lambda})\}$$
		such that
		\begin{enumerate}[(i)]
			\item $(\rho_{\pi,\lambda},\varepsilon_{\ell}^{1-n}\rho_{\chi,\lambda})$ is essentially self-dual and totally odd, where $\varepsilon_{\ell}$ is the $\ell$-adic cyclotomic character and $\lambda|\ell$.
			\item Fix an embedding $\iota:E_{\lambda}\hookrightarrow\mathbb{C}$. For $v\nmid\ell$, the semisimplified Weil-Deligne representation is independent of $\lambda$ and satisfies:
			$$\iota\mathrm{WD}(\rho_{\pi,\lambda}|_{\mathrm{Gal}_{F_v}})^{\mathrm{F-ss}}\cong\mathrm{rec}\left(\pi_v\otimes|\det|_v^{(1-n)/2}\right)$$
			and these Weil-Deligne representations are pure of weight $\omega$.
			\item $\rho_{\pi,\lambda}$ is de Rham, has pure of weight $\omega$ and distinct $\tau$-Hodge-Tate weights for all $\tau:F\hookrightarrow\overline{E}$.
			\item If $v|\ell$ and $\pi_v$ has an Iwahori fixed vector then
			$$\iota\mathrm{WD}(\rho_{\pi,\lambda}|_{\mathrm{Gal}_{F_v}})^{\mathrm{F-ss}}\cong\mathrm{rec}\left(\pi_v\otimes|\det|_v^{(1-n)/2}\right)$$
			In particular $\rho_{\pi,\lambda}$ is semi-stable at $v$, and if $\pi_v$ is unramified then it is crystalline.
		\end{enumerate}
	\end{thm}
	
	One has following criterion on automorphic Galois representations of totally real fields.
	\begin{thm}\cite[Theorem C]{BLGGT14}\label{Lithium}.
		Suppose $K$ is a totally real field. Let $n$ be an integer and $\ell\ge 2(n+1)$ be a prime. Let
		$$\rho:\mathrm{Gal}_K\to\text{GL}_n(\overline{\mathbb{Q}_{\ell}})$$
		be a continuous representation. Suppose that the following conditions are satisfies.
		\begin{enumerate}[(i)]
			\item (Unramified almost everywhere) $\rho$ is unramified at all but finitely many primes.
			\item (Odd essential self-duality) Either $\rho$ maps to $\text{GSp}_n$ with totally odd similitude character or it maps to $\text{GO}_n$ with totally even similitude character.
			\item (Potential diagonalizability and regularity) $\rho$ is potentially diagonalizable (and hence potentially crystalline) at each prime $v$ of $K$ above $\ell$ and regular, i.e. for each $\tau:K\hookrightarrow\overline{\mathbb{Q}_{\ell}}$ it has $n$ distinct $\tau$-Hodge-Tate weights.
			\item (Irreducibility) $\rho|_{\mathrm{Gal}_{K(\zeta_{\ell})}}$ is residually irreducible.
		\end{enumerate}
		Then we can find a finite Galois totally real extension $K'/K$ such that $\rho|_{\mathrm{Gal}_{K'}}$ is automorphic. Moreover $\rho$ is part of a strictly pure compatible system of $K$.
	\end{thm}
	
	Condition (iii) can be checked by following.
	\begin{lem}\label{Sodium}
		When $K=\mathbb{Q}$, condition (iii) is satisfied when $\rho$ is crystalline and regular, and the Hodge-Tate numbers $\mathrm{Ht}(\rho)\subseteq[a,a+\ell-2]$ for some integer $a$.
	\end{lem}
	\begin{proof}
		One takes $K=\mathbb{Q}$ in \cite[Lemma 1.4.3.(2)]{BLGGT14}.
	\end{proof}
	
	The following result shows that certain low dimensional subrepresentations of strictly compatible systems fit into strictly compatible systems.
	\begin{prop}\cite[Proposition 2.12]{Hu23a}\label{Boron}.
		Given an $E$-rational strictly compatible system $\{\rho_{\lambda}\}$ of some totally real field, then for all but finitely many $\lambda$, 
		\begin{enumerate}[(i)]
			\item If $\sigma$ is a 2-dimensional irreducible regular subrepresentation of $\rho_{\lambda}$, then $\sigma$ can be extended to a 2-dimensional regular irreducible strictly compatible system.
			\item If $\sigma$ is a 3-dimensional irreducible regular essentially self-dual subrepresentation of $\rho_{\lambda}$, then $\sigma$ can be extended to a 3-dimensional regular irreducible strictly compatible system.
		\end{enumerate}
	\end{prop}
	
	One has the following criterion to check an irreducible subrepresentation of $\rho_{\pi,\lambda}$ is essentially self-dual and totally odd.
	\begin{thm}\cite[Theorem 2.3]{CG13}\label{Potassium}
		Let $(\pi,\chi)$ be a regular algebraic cuspidal polarized automorphic representation of $\mathrm{GL}_n(\mathbb{A}_F)$ where $F$ is a totally real field. Denote by $(\rho_{\pi,\lambda},\rho_{\chi,\lambda})$ the corresponding compatible system of Galois representations. If for some $\lambda$ we have an irreducible subrepresentation $r$ of $\rho_{\pi,\lambda}$ such that $r\cong r^{\vee}\otimes\varepsilon_{\ell}^{1-n}\rho_{\chi,\lambda}$, where $\lambda|\ell$ and $\varepsilon_{\ell}$ is the $\ell$-adic cyclotomic character, then $(r,\varepsilon_{\ell}^{1-n}\rho_{\chi,\lambda})$ is essentially self-dual and totally odd.
	\end{thm}
	
	In particular, for dimensional reason one has the following consequence.
	\begin{cor}\label{Calcium}
		Under the above setting, if the irreducible components of some $\rho_{\pi,\lambda}$ have distinct dimensions, then each of them is essentially self-dual and totally odd.
	\end{cor}
	
	\subsection{Big images and irreducibility}
	
	In the sequel we denote by $(\overline{\rho}^{\text{ss}},\overline{V}^{\text{ss}})$ the semisimple reduction of a $\lambda$-adic Galois representation $(\rho,V)$, by $\varepsilon_{\ell}$ the $\ell$-adic cyclotomic character of some number field.
	\begin{defn/thm}\cite[Theorem 3.1]{Hu23b},\cite[Theorem 2.10]{Hu23a}.\label{Sulfur}
		Given an $n$-dimensional regular $E$-rational semisimple Serre compatible system $\{(\rho_{\lambda},V_{\lambda})\}$ of number field $K$. Write $d=[E:\mathbb{Q}]$. By restriction of scalars, we have an $nd$-dimensional $\mathbb{Q}$-rational compatible system:
		$$\left\{\rho_{\ell}:=\bigoplus_{\lambda|\ell}\rho_{\lambda}:\mathrm{Gal}_K\to\left(\mathrm{Res}_{E/\mathbb{Q}}\right)\left(\mathbb{Q}_{\ell}\right)\subseteq\mathrm{GL}_{nd}(\mathbb{Q}_{\ell})\right\}_{\ell}$$
		Suppose that there exist integers $N_1.N_2\ge0$ and a finite extension $K'/K$ such that the following conditions hold.
		\begin{enumerate}[(a)]
			\item (Bounded tame inertia weights): For all $\ell\gg0$ and each finite place $v$ of $K$ above $\ell$, the tame inertia weights of the local representation $(\overline{\rho}^{\mathrm{ss}}_{\ell}\otimes\overline{\varepsilon}^{N_1}_{\ell})|_{\mathrm{Gal}_{K_v}}$ belong to $[0,N_2]$.
			\item (Potential semistability): For all $\ell\gg0$ and each finite place $w$ of $K'$ not above $\ell$, the semisimplification of the local representation $\overline{\rho}^{\text{ss}}$.
		\end{enumerate}
		Then there exists a finite Galois extension $L/K$ such that, up to isomorphism there exists a unique connected reductive group $$\underline{G}_{\ell}\subseteq\text{GL}_{nd,\mathbb{F}_{\ell}}$$
		for each sufficiently large $\ell$ called algebraic envelope, such that:
		\begin{enumerate}[(i)]
			\item $\overline{\rho_{\ell}}^{\text{ss}}(\mathrm{Gal}_L)$ is a subgroup of $\underline{G}_{\ell}(\mathbb{F}_{\ell})$ with index uniformly bounded when $\ell$ varies.
			\item $\underline{G}_{\ell}$ acts on the ambient space semisimply.
			\item The formal characters of $\underline{G}_{\ell}\hookrightarrow\text{GL}_{nd,\mathbb{F}}$ for all $\lambda$ are bounded.
		\end{enumerate}
	\end{defn/thm}
	
	For all but finitely many $\ell$ such that the algebraic envelope $\underline{G}_{\ell}$ exists, let $\lambda\in\Sigma_E$ be any finite place of $E$ that divides $\ell$ and $(\sigma, W)$ be a subrepresentation of $\rho_{\lambda}\otimes\overline{\mathbb{Q}}_{\ell}$. Denote by $\underline{G}_W$ the image of $\underline{G}_{\ell}$ in $\text{GL}_{\overline{W}^{\text{ss}}}$, which is called algebraic envelope of $W$.
	\begin{thm}\cite[Theorem 3.12]{Hu23b}.\label{Beryllium}
		Given an $n$-dimensional $E$-rational semisimple Serre compatible system $\{\rho_{\lambda}\}$ of number field $K$. Assume the conditions (a) and (b) in \autoref{Sulfur} hold. Then except for finitely many $\lambda$, for any subrepresentation $(\sigma,W)$ of $\rho_{\lambda}\otimes\overline{E}_{\lambda}$ one has:
		\begin{enumerate}[(i)]
			\item The algebraic envelope $\underline{G}_W$ and algebraic monodromy $\textbf{G}_W$ of $\sigma$ have the same formal bi-characters.
			\item There exists a finite Galois extension $L/K$, independent of $W$, such that the commutants of $\overline{\sigma}_{\lambda}^{\text{ss}}(\mathrm{Gal}_L)$ and $\underline{G}_{W}$ (resp. $[\overline{\sigma}_{\lambda}^{\text{ss}}(\mathrm{Gal}_L),\overline{\sigma}_{\lambda}^{\text{ss}}(\mathrm{Gal}_L)]$ and $\underline{G}_{W}^{\text{ss}}$) in $\text{End}(\overline{W})^{\text{ss}}$ are equal. In particular, $\overline{\sigma}_{\lambda}^{\text{ss}}(\mathrm{Gal}_L)$ (resp. $[\overline{\sigma}_{\lambda}^{\text{ss}}(\mathrm{Gal}_L),\overline{\sigma}_{\lambda}^{\text{ss}}(\mathrm{Gal}_L)]$) is irreducible on $\overline{W}^{\text{ss}}$ if and only if $\underline{G}_{W}$ (resp. $\underline{G}_{W}^{\text{der}}$) is irreducible on $\overline{W}^{\text{ss}}$.
			\item If $\textbf{G}_W$ is of type A and $\textbf{G}_W^{\circ}\to\mathrm{GL}_W$ is irreducible (in particular for Lie-irreducible dimension $\le3$ ones), then $\underline{G}_W$ and thus $\mathrm{Gal}_K$ (resp. $\mathrm{Gal}_{K^{\text{ab}}}$) are irreducible on $\overline{W}^{\mathrm{ss}}$.
			\item If $\sigma$ is irreducible and of type A, then it is residually irreducible.
		\end{enumerate}
	\end{thm}
	
	Given an $E$-rational regular strictly compatible system. By \autoref{Phosphorus}, after enlarging $E$, one regards the system as a Serre compatible system $\{\rho_{\lambda}:\mathrm{Gal}_K\to\mathrm{GL}_n(E_{\lambda})\}$. 
	\begin{thm}\cite[Theorem 4.1]{Hu23b}
		The conclusions in \autoref{Beryllium} hold for $E$-rational regular strictly compatible systems, in particular for compatible systems arising from \autoref{Silicon}.
	\end{thm}
	
	We need following result in the proof.
	\begin{prop}\cite[Proposition 2.25]{Da25}\label{Carbon}
		Let $\{\rho_{\lambda}:\mathrm{Gal}_{\mathbb{Q}}\to\text{GL}_n(\overline{E}_{\lambda})\}$ be an $E$-rational strictly compatible system of $\mathbb{Q}$. Consider its modulo $\lambda$ compatible system $\{\overline{\rho}_{\lambda}^{\text{ss}}\}$ by taking semisimple reductions. Suppose for infinitely many $\lambda$ one has a 2-dimensional odd irreducible subrepresentation
		$$\overline{\sigma}_{\lambda}\subseteq\overline{\rho}_{\lambda}^{\text{ss}}$$
		Then after replacing $E$ with a finite extension, there exists a 2-dimensional $E$-rational strictly compatible system $\{\sigma_{\lambda}\}$ such that for infinitely many $\lambda$ one has $\overline{\sigma}_{\lambda}$ is the semisimple reduction of $\sigma_{\lambda}$.
	\end{prop}
	
	\subsection{$p$-adic Hodge theoretic lift}
	
	Let $\widetilde{H}\twoheadrightarrow H$ be a central torus quotient of algebraic groups. Let $F$ be a global or local field. The following results lift an $\ell$-adic representation $\rho:\text{Gal}_F\to H(\overline{\mathbb{Q}}_{\ell})$ to $\widetilde{\rho}:\text{Gal}_F\to\widetilde{H}(\overline{\mathbb{Q}}_{\ell})$ that preserve certain $p$-adic Hodge properties.
	\begin{equation}\label{2}
		\xymatrix{&&\widetilde{H}(\overline{\mathbb{Q}}_{\ell})\ar@{->>}[d]\\\mathrm{Gal}_{F}\ar[rr]_{\rho}\ar@{-->}[urr]^{\widetilde{\rho}}&&H(\overline{\mathbb{Q}}_{\ell})}
	\end{equation}
	
	\begin{thm}\cite[Corollary 3.2.12]{Pa19}\label{Magnesium}
		Let $\widetilde{H}\twoheadrightarrow H$ be a central torus quotient of algebraic groups, and let $\rho:\mathrm{Gal}_{F}\to H(\overline{\mathbb{Q}}_{\ell})$ be a Hodge-Tate representation of a local field $F$. Then there exists a Hodge-Tate representation $\widetilde{\rho}:\mathrm{Gal}_{F}\to \widetilde{H}(\overline{\mathbb{Q}}_{\ell})$ such that (\ref{2}) commutes.
	\end{thm}
	
	\begin{thm}\label{Helium}\cite[Theorem 2.13]{DWW24}
		Let $\widetilde{H}\twoheadrightarrow H$ be a central torus quotient of algebraic groups, and let $\rho:\mathrm{Gal}_{\mathbb{Q}}\to H(\overline{\mathbb{Q}}_{\ell})$ be an $\ell$-adic representation of $F=\mathbb{Q}$ that is unramified almost everywhere and the restriction to $\text{Gal}_{\mathbb{Q}_{\ell}}$ is crystalline. Then there exists a representation $\widetilde{\rho}:\mathrm{Gal}_{F_v}\to\widetilde{H}(\overline{\mathbb{Q}}_{\ell})$ that is unramified almost everywhere and the restriction to $\text{Gal}_{\mathbb{Q}_{\ell}}$ is crystalline, such that (\ref{2}) commutes.
	\end{thm}
	\begin{proof}
		By \cite[Proposition 5.5]{Pa15} under the setting one has a geometric lift $\widetilde{\rho}'$ of $\rho$. Hence one just needs to modify this lift such that its restriction to $\text{Gal}_{\mathbb{Q}_{\ell}}$ is moreover crystalline. This can be done by \cite[Corollary 3.2.13]{Pa19}, which says locally one has a lift $\tau:\text{Gal}_{\mathbb{Q}_{\ell}}\to\widetilde{H}(\overline{\mathbb{Q}}_{\ell})$ of $\rho$ that is crystalline. As both being a lift of $\rho|_{\text{Gal}_{\mathbb{Q}_{\ell}}}$, one has $\tau\cong\widetilde{\rho}'|_{\text{Gal}_{\mathbb{Q}_{\ell}}}\otimes\chi$ for some Hodge-Tate character $\chi$. One can twist $\tau$ with suitable power of $\varepsilon_{\ell}$ making the Hodge-Tate weight of $\chi$ to be zero. In particular the restriction to inertia subgroup $\chi|_{I_{\ell}}$ has finite image. Hence one can choose some global character $\chi':\text{Gal}_{\mathbb{Q}}\to\overline{Q}_{\ell}^*$ such that $\chi'|_{I_{\mathbb{Q}_{\ell}}}=\chi|_{I_{\mathbb{Q}_{\ell}}}$. Then $\widetilde{\rho}=\widetilde{\rho}'\otimes\chi'$ is the desired lift of $\rho$.
	\end{proof}
	
	\section{A step of Xia}
	
	It is known that at least for infinitely many $\lambda$, one has $\rho_{\pi,\lambda}$ is irreducible.
	\begin{thm}\cite[Theorem 1.7]{PT15}\label{Chlorine}
		Let $F$ be a CM field and $\pi$ is a polarisable, regular algebraic, cuspidal automorphic representation of $\text{GL}_n(\mathbb{A}_F)$. Denote by $E_{\pi}$ the number field that each $\rho_{\pi,\lambda}$ is defined over, whose existence is guaranteed by \autoref{Phosphorus}. Then there is a finite CM extension $E/E_{\pi}$ and a Dirichlet density 1 set $\mathcal{L}$ of rational primes, such that for all conjugation-invariant primes $\lambda$ of $E$ dividing an $\ell\in\mathcal{L}$, $\rho_{\pi,\lambda}|_{E_{\pi}}$ is irreducible. In particular, there is a positive Dirichlet density set $\mathcal{L}'$ of rational primes such that if a prime $\lambda$ of $E_{\pi}$ divides some $\ell\in\mathcal{L}'$, then $\rho_{\pi,\lambda}$ is irreducible.
	\end{thm}
	
	In this section we follow the treatment in \cite{Xi19} to prove the following result.
	\begin{prop}\label{Fluorine}
		To prove \autoref{main} it is enough to assume there exist infinitely many places $\lambda$ such that $\rho_{\pi,\lambda}$ is Lie-irreducible.
	\end{prop}
	
	Notice that when $n=7$, this is easy to verify. Pick one irreducible $\rho_{\pi,\lambda_0}$. If this is not Lie-irreducible, then it is induced by a character $\chi_{\lambda_0}$ of a 7-degree numer field $K$. By class field theory, after possibly enlarging the coefficients, this $\chi_{\lambda_0}$ extends to a strictly compatible system $\{\chi_{\lambda}\}$. Then by semisimplicity of $\{\rho_{\pi,\lambda}\}$, one has $\rho_{\pi,\lambda}=\mathrm{Ind}_K^{\mathbb{Q}}\chi_{\lambda}$. Then one uses regularity to check conditions in Mackey's irreducibility criterion, hence in this case \textbf{all} $\rho_{\pi,\lambda}$ are irreducible.
	
	We focus on $n=8$. One has the following result which is a totally real analogy to \cite[Proposition 2]{Xi19}.
	\begin{prop}\cite[Proposition 4.14]{Hu23b}
		Let $F^+$ be a totally real field, $\{\rho_{\pi,\lambda}\}$ the associated compatible system of $F^+$ defined over $E$ as in \autoref{Phosphorus}. Let $F$ be a CM field containing $F^+$ as maximal totally real subfield. Let $F_{1,\pi}$ be the minimal extension of $F^+$ such that the compatible system
		$$\left\{\left(\mathrm{Ind}_F^{F^+}\mathrm{Res}_F^{F^+}\rho_{\pi,\lambda}\right)\oplus\rho_{\chi,\lambda}\right\}$$
		is connected. Let $F_2$ be the maximal CM subextension of $F_{1,\pi}/F^+$. After enlaging the CM field $E$ if necessary, there exist a family of Galois representations $\{r_{1,\lambda}\}_{\lambda}$ of a subextension $F_4$ of $F_2/F^+$ and a regular algebraic polarized cuspidal automorphic representation $\pi_1$ of $\mathrm{GL}_m(\mathbb{A}_{F_3})$ where $F_3$ is a finite CM extension of $F_2$ such that
		$$\{\mathrm{Ind}_{F_4}^{F^+}r_{1,\lambda}\}_{\lambda}\cong\{\rho_{\pi,\lambda}\}_{\lambda}\ \mathrm{and}\  \{\mathrm{Res}_{F_3}^{F_4}r_{1,\lambda}\}_{\lambda}\cong\{\rho_{\pi_1,\lambda}\}_{\lambda}$$
		and this $F_3$ is the maximal CM subextension of $F_{1,\pi}/F^+$.
	\end{prop}
	
	If $F_4\not=F^+$, then $m\le4$. Hence $\rho_{\pi_1,\lambda}$ is irreducible for all but finitely many $\lambda$ due to \cite[Theorem 1.4]{Hu23b}. Then by regularity and Mackey's irreducibility criterion we have $\{\rho_{\pi,\lambda}\}$ is irreducible for all but finitely many $\lambda$. Hence we can assume $F_4=F^+$. Then $r_{1,\lambda}\cong\rho_{\pi,\lambda}$ by semisimplicity and \autoref{Fluorine} follows from below.
	\begin{prop}\cite[Corollary 1]{Xi19}
		Let $F$ be a CM field and $\{\rho_{\pi,\lambda}\}_{\lambda}$ be the compatible system associated to $\pi$. If $F$ is maximal CM subextension of $F_{1,\pi}/F^+$. Then there exist infintely many places $\lambda$ such that $\rho_{\pi,\lambda}$ is Lie-irreducible.
	\end{prop}
	
	\section{proof}
	
	For simplicity we omit $\pi$ in each associated representation $\rho_{\pi,\lambda}$. By \autoref{Fluorine} we assume there are infinitely many $\lambda_0$ such that $\rho_{\lambda_0}$ is Lie-irreducible. Then the restriction to its derived subgroup $\textbf{G}_{\lambda_0}^{\mathrm{der}}=[\textbf{G}_{\lambda}^{\circ},\textbf{G}_{\lambda}^{\circ}]$, which we denote by $\rho_{\lambda_0}^{\mathrm{der}}$, is irreducible.
	
	\begin{prop}\label{Argon}
		The list \autoref{table1} gives all the isomorphism classes of connected semisimple subgroups $G\subseteq\mathrm{GL}_V$ that are irreducible on $V=\overline{\mathbb{Q}}_{\ell}^n$ for $n=7$ and 8.
	\end{prop}
	\begin{proof}
		The tautological representation $\rho$ of $G$ admits an exterior tensor decomposition
		$$(G,\rho)\cong(G_1G_2\cdots G_m,\rho_1\otimes\rho_2\otimes\cdots\otimes\rho_m)$$
		where each $G_i$ is an almost simple factor of $G$ and $(G_i,\rho_i)$ is an irreducible representation. Then one tracks down the low dimensional irreducible representations of almost simple lie algebra gives the complete list.
		\begin{enumerate}[(i)]
			\item $n=7$ and $m=1$: cases (1), (2), (3), (4).
			\item $n=8$ and $m=1$: cases (5), (7), (10), (12), (13), (14), (15).
			\item $n=8$ and $m=2$: cases (6), (9), (11).
			\item $n=8$ and $m=3$: case (8) only.
		\end{enumerate}
	\end{proof}
	
	Given a Lie-irreducible Galois representation $\rho$, by Lie type of $\rho$, we mean the isomorphism class of the tautological representation of $\rho^{\text{der}}$. As there are only finitely many Lie types for $\rho_{\lambda_0}$, we can assume there exists an infinite set of places $L\subseteq \Sigma_E$ such that all $\rho_{\lambda_0}^{\mathrm{der}}$ for $\lambda_0\in L$ have the same Lie type listed in \autoref{table1}. As $\text{Im}\rho_{\lambda}$ must factor through $\text{GO}_n$ or $\text{GSp}_n$ due to \autoref{Silicon}.(i) and \autoref{Scandium}.(iii), we rule out cases (4), (11) and (15). In the sequel we assume there exist infinitely many $\lambda_1$ such that $\rho_{\lambda_1}$ is reducible.
	\begin{center}
		\label{table1}
		\centering
		\begin{tabular}{|l|l|l|c|c|c|}
			\hline
			&Types & (\text{G},V) & dim & rank & formal character\\
			\hline
			(1)&$7A_1$ & $(\mathrm{SL}_2,\mathrm{Sym}^6(\mathrm{Std}))$ & 7 & 1& $\{x^{-3},x^{-2},x^{-1},1,x,x^2,x^3\}$\\
			\hline 
			(2)&$7G_2$ &$(\mathrm{G}_2,\mathrm{Std})$&7&2&$\{x,x^{-1},y,y^{-1},xy,(xy)^{-1},1\}$\\
			\hline 
			(3)&$7B_3$ &$(\mathrm{SO}_7,\mathrm{Std})$&7&3&$\{x,x^{-1},y,y^{-1},z,z^{-1},1\}$\\
			\hline 
			(4)&$7A_6$ &$(\mathrm{SL}_7,\mathrm{Std})$&7&6&omitted\\
			\hline 
			(5)&$8A_1$ &$(\mathrm{SL}_2,\mathrm{Sym}^7(\mathrm{Std}))$&8&1&$\{x^{-7},x^{-5},x^{-3},x^{-1},x,x^3,x^5,x^7\}$\\
			\hline 
			(6)&$2A_1\times4A_1$ &$(\mathrm{SL}_2\times\mathrm{SL}_2,\mathrm{Std}\otimes\mathrm{Sym}^3(\mathrm{Std}))$&8&2&omitted\\
			\hline 
			(7)&$8A_2$ &$(\mathrm{SL}_3,\mathrm{adjoint\ representation})$&8&2&omitted\\
			\hline 
			(8)&$2A_1\times2A_1\times2A_1$ &$(\mathrm{SL}_2\times\mathrm{SL}_2\times\mathrm{SL}_2,\mathrm{Std}\otimes\mathrm{Std}\otimes\mathrm{Std})$&8&3&$\{x^{\pm1}y^{\pm1}z^{\pm1}\}$\\
			\hline 
			(9)&$2A_1\times4C_2$ &$(\mathrm{SL}_2\times\mathrm{Sp}_4,\mathrm{Std}\otimes\mathrm{Std})$&8&3&$\{x^{\pm1}y^{\pm1}z^{\pm1}\}$\\
			\hline 
			(10)&$8B_3$&$(\mathrm{SO}_7,\mathrm{Spin\ representation})$&8&3&$\{x^{\pm1}y^{\pm1}z^{\pm1}\}$\\
			\hline
			(11)&$2A_1\times4A_3$ &$(\mathrm{SL}_2\times\mathrm{SL}_4,\mathrm{Std}\otimes\mathrm{Std})$&8&4&omitted\\
			\hline 
			(12)&$8C_4$ &$(\mathrm{Sp}_8,\mathrm{Std})$&8&4&$\{x,x^{-1},y,y^{-1},z,z^{-1},w,w^{-1}\}$\\
			\hline 
			(13)&$8D_4$ &$(\mathrm{SO}_8,\mathrm{Std})$&8&4&$\{x,x^{-1},y,y^{-1},z,z^{-1},w,w^{-1}\}$\\
			\hline
			(14)&$8D_4$&$(\mathrm{SO}_8,\mathrm{two\ half-spin\ representations})$&8&4&omitted\\
			\hline 
			(15)&$8A_7$ &$(\mathrm{SL}_8,\mathrm{Std})$&8&7&omitted\\
			\hline 
		\end{tabular}
	\end{center}
	
	\subsection{Case (1)}
	
	Assume $\rho_{\lambda_0}^{\mathrm{der}}$ is of type $7A_1$. As the formal bi-character of $\rho_{\lambda_1}$ is the same as that of $\rho_{\lambda_0}$ due to \autoref{Hydrogen}.(ii), the decompositions of $\rho_{\lambda_1}^{\mathrm{der}}$ can only be
	$$\rho_{\lambda_1}^{\mathrm{der}}=\left(\mathrm{SL}_2,\mathrm{Sym}^2(\mathrm{Std})\right)\oplus\left(\mathrm{SL}_2,\mathrm{Sym}^3(\mathrm{Std})\right)$$
	
	We denote by $\rho_{\lambda_1}=\sigma_{\lambda_1,1}\oplus\sigma_{\lambda_1,2}$ the irreducible decomposition. We check conditions in \autoref{Lithium} to show both $\sigma_{\lambda_1,1}$ and $\sigma_{\lambda,2}$ extend to a compatible system for some $\lambda_1$. Then the semisimplicity of $\{\rho_{\lambda}\}$ would give $\rho_{\lambda_0}$ is reducible hence a contradiction. Condition (i) is obvious. As $\rho_{\lambda_1}$ is essential self-dual and odd, and $\sigma_{\lambda_1,1}$, $\sigma_{\lambda_1,2}$ have different dimensions, (ii) can be checked by \autoref{Calcium}. (iii) can be checked by \autoref{Sodium} after taking $\lambda_1$ sufficiently large. Finally, as both $\sigma_{\lambda_1,i}$ are of type A, by \autoref{Beryllium}.(iii), the last condition (iv) is satisfied for $\lambda_1$ sufficiently large. Hence in such case $\rho_{\lambda}$ is irreducible for all but finitely many $\lambda$.
	
	\subsection{Case (3)}
	
	Assume $\rho_{\lambda_0}^{\mathrm{der}}$ is of type $7B_3$ with standard representation. By \autoref{Hydrogen}.(ii), the formal character of each $\rho_{\lambda}^{\text{der}}$ is $\{x,x^{-1},y,y^{-1},z,z^{-1},1\}$. Since there exists only one zero weight, there cannot be more than one character in the decomposition of $\rho_{\lambda_1}$. Hence there exist infinitely many $\lambda_1$ such that $\rho_{\lambda_1}$ all have dimension type one of the following cases:
	\begin{enumerate}[(i)]
		\item $6+1$.
		\item $\rho_{\lambda_1}$ has a 2 or 3-dimensional component.
	\end{enumerate}
	
	We first consider case (ii). The 2 or 3-dimensional component $\varphi_{\lambda_1}$ of $\rho_{\lambda_1}$ must be Lie-irreducible. Since otherwise the derived subgroup of its algebraic monodromy group is trivial, which contradicts with the fact that the formal character of $\rho_{\lambda}^{\text{der}}$ has no repeated zero weights.
	
	As the formal character of $\rho_{\lambda_0}^{\text{der}}$ contains no three nonzero weights such that their sum is zero, the 3-dimensional Lie-irreducible component of $\rho_{\lambda_1}$ (if exists) must be of type $\mathrm{SO}_3$. Then by \autoref{Boron} for some $\lambda_1$ our $\varphi_{\lambda_1}$ fits into a strictly compatible system $\{\varphi_{\lambda}\}$. The derived subgroup of algebraic monodromy group of $\varphi_{\lambda_1}$ is either $\mathrm{SL}_2$ or $\mathrm{SO}_3$. Consider compatible system $\{\rho_{\lambda}\oplus\varphi_{\lambda}\}$. At place $\lambda_1$ the semisimple rank of $\rho_{\lambda_1}\oplus\varphi_{\lambda_1}$ is 3 since $\varphi_{\lambda_1}$ is a subrepresentation of $\rho_{\lambda_1}$. At place $\lambda_0$, by Goursat's lemma, the derived subgroup of monodromy group of $\rho_{\lambda_0}\oplus\varphi_{\lambda_0}$ is either $\mathrm{SO}_7\times\mathrm{SL}_2$ or $\mathrm{SO}_7\times\mathrm{SO}_3$, which has rank 4. This contradicts with \autoref{Hydrogen}.(ii).
	
	For case (i), denote by $\varphi_{\lambda_1}\oplus\chi_{\lambda_1}$ the irreducible decomposition of $\rho_{\lambda_1}$, where $\varphi_{\lambda_1}$ is a 6-dimensional component and $\chi_{\lambda_1}$ is a character. Since each $\rho_{\lambda}$ factors through $\mathrm{GO}_7$ and $\mathrm{Sp}_6\times\{1\}$ is not contained in $\mathrm{GO}_7$, the component $\varphi_{\lambda_1}$ must factor through $\mathrm{GO}_6$. Since $\rho_{\lambda_1}$ is essentially self-dual and odd, and the components of $\rho_{\lambda_1}$ have different dimensions, by \autoref{Calcium}, $\varphi_{\lambda_1}$ is essentially self-dual and odd. Since $\mathrm{SO}_6$ is of type A, \autoref{Beryllium}.(iii) shows that there exists some $\lambda_1$ such that (iv) of \autoref{Lithium} is true. Other conditions of the theorem are easy to verify. Hence this $\varphi_{\lambda_1}$ fits into a strictly compatible system. The character $\chi_{\lambda_1}$ naturally fits into a compatible system (after possibly enlarging the coefficients) due to class field theory. Hence semisimplicity of $\{\rho_{\lambda}\}$ implies $\rho_{\lambda_0}$ is reducible, a contradiction.
	
	\subsection{Case (5)}
	
	Assume $\rho_{\lambda_0}^{\mathrm{der}}$ is of type $8A_1$. Since the formal character of $\rho_{\lambda_0}^{\mathrm{der}}$ does not contain zero weight, it cannot be decomposed as the union of two formal characters of some representations. Hence by \autoref{Hydrogen}.(ii) all $\rho_{\lambda}$ are irreducible with same Lie type. 
	
	\subsection{Case (7)}
	
	Assume $\rho_{\lambda_0}^{\mathrm{der}}$ is the adjoint representation of $\mathrm{SL}_3$. We show in this case $\rho_{\lambda_0}$ is irregular hence rule out this situation.
	
	Denote by $K$ the smallest extension of $\mathbb{Q}$ as in \autoref{Hydrogen}.(i). We restrict the compatible system to $\text{Gal}_K$. Then all algebraic monodromy groups are connected. Since the image of adjoint representation is $\mathrm{PGL}_3$, the algebraic monodromy group $\textbf{G}_{\lambda_0}$ is either $\text{PGL}_3$ or $\mathbb{G}_m\cdot\text{PGL}_3$. The $\mathbb{G}_m$ part corresponds to a character which is a weakly abelian direct summand of $\rho_{\lambda_0}$, hence by \cite[Theorem 1.1]{BH25} this character is Hodge-Tate. Hence after twisting a compatible system of Hodge-Tate characters to $\{\rho_{\lambda}\}$, we may assume $\textbf{G}_{\lambda_0}=\mathrm{PGL}_3$.
	
	Consider the surjection $\mathrm{GL}_3\to\mathrm{PGL}_3$ whose kernel is a central torus. By \autoref{Magnesium} the restriction of $\rho_{\lambda_0}$ to $\mathrm{Gal}_{\mathbb{Q}_{\ell}}$ can lift to some Hodge-Tate representation $\sigma$:
	$$\xymatrix{&&\mathrm{GL}_3(\overline{E}_{\lambda})\ar[d]\\\mathrm{Gal}_{\mathbb{Q}_{\ell}}\ar@{-->}[urr]^{\sigma}\ar[rr]_{\rho_{\lambda_0}|_{\mathrm{Gal}_{\mathbb{Q}_{\ell}}}}&&\mathrm{PGL}_3(\overline{E}_{\lambda})}$$
	Then there exists some characters $\chi_1$ and $\chi_2$ such that
	$$\sigma\otimes\left(\sigma^{\vee}\otimes\chi_1\right)\cong\chi_2\oplus\rho_{\lambda_0}|_{\mathrm{Gal}_{\mathbb{Q}_{\ell}}}$$
	In particular $\rho_{\lambda_0}$ has repeated Hodge-Tate weights.
	
	\subsection{Cases (10), (12), (13), (14)}
	
	As the formal character of $\rho_{\lambda_0}^{\text{der}}$ contains no zero weights and no three weights whose sum is zero, $\rho_{\lambda_1}$ cannot contain 1 or 3 or 5-dimensional components. We first show the following lemma.
	\begin{lem}\label{Titanium}
		Given an 8-dimensional compatible system $\{\rho_{\lambda}\}$ which is attached to a regular algebraic essentially self-dual cuspidal automorphic representation $\pi$ of $\text{GL}_n(\mathbb{A}_\mathbb{Q})$ such that at least one $\rho_{\lambda_0}$ is Lie-irreducible. If some $\rho_{\lambda_1}$ is reducible with dimensional type $4+4$, then exactly one of the following happens.
		\begin{enumerate}[(i)]
			\item The two 4-dimensional components are essentially self-dual and odd.
			\item Both components are not essentially self-dual. In such case one has
			\begin{equation}\label{1}
				\textbf{G}_{\lambda_1}^{\mathrm{der}}=(\mathrm{SL}_4,\mathrm{Std})\oplus(\mathrm{SL}_4,\mathrm{Std}^{\vee})
			\end{equation}
		\end{enumerate}
	\end{lem}
	\begin{proof}
		We assume the irreducible decomposition of $\rho_{\lambda_1}$ is $W_1\oplus W_2$ where $\dim W_1=\dim W_2=4$. Denote by $\chi=\varepsilon_{\ell}^{7}\rho_{\chi,\lambda_1}^{-1}$. Then by \autoref{Silicon}.(i) one has $\rho_{\lambda_1}^{\vee}\cong\rho_{\lambda_1}\otimes\chi$. Then for dimensional reason we have two cases.
		\begin{enumerate}[(a)]
			\item $W_i^{\vee}\cong W_i\otimes\chi$ for $i=1,2$.
			\item $W_1^{\vee}\cong W_2\otimes\chi$ and $W_2^{\vee}\cong W_1\otimes\chi$.
		\end{enumerate}
		
		In case (a), due to \autoref{Potassium}, we have both $W_1$ and $W_2$ are essentially self-dual and odd. Hence this is in case (i). In case (b), if one of them is not Lie-irreducible, say $W_1$, one writes it as
		$$W_1=\text{Ind}^{\mathbb{Q}}_K\sigma$$
		for some Lie-irreducible representation $\sigma$ of $K$. Then $W_2$ would also not be Lie-irreducible.
		$$W_2=\text{Ind}^{\mathbb{Q}}_K(\sigma^{\vee}\otimes\chi|_{\text{Gal}_K}^{-1})$$
		Here $\sigma^{\vee}\otimes\chi|_{\text{Gal}_K}^{-1}$ also is Lie-irreducible. Hence we have some $\rho_{\lambda_0}$ is Lie-irreducible yet $\rho_{\lambda_1}$ is induced by some representation of a nontrivial field extension. This contradicts with the compatibility of Frobenii due to \cite[Proposition 3.4.9]{Pa19} (see also \cite[Proposition 2.15]{Da25}).
		
		Hence both $W_1$ and $W_2$ are Lie-irreducible, then their Lie type are either $(\text{SL}_2,\text{Sym}^3(\text{Std}))$, $(\text{SO}_4,\text{Std})$, $(\text{Sp}_4,\text{Std})$ or $(\text{SL}_4,\text{Std})$. We rule out case $\text{SL}_2$ since the semisimple rank of algebraic monodromy group of $\rho_{\lambda_1}$ is 1, hence the Lie-irreducible $\rho_{\lambda_0}$ must have Lie type $(5)$ as in \autoref{table1}. But we have shown in this case all $\rho_{\lambda}$ are irreducible.
		
		If the Lie type is $\text{SL}_4$, then both $W_1$ and $W_2$ are not essentially self-dual and we are in case (ii).
		
		If the Lie type is $\text{Sp}_4$ or $\text{SO}_4$, then both $W_1$ and $W_2$ are essentially self-dual. We write $\eta$ a similitude character of $W_1$. Then $W_1^{\vee}\cong W_1\otimes\eta$. Hence $W_2\cong W_1\otimes\eta\chi^{-1}$. However this would implies each weight in the formal character of $\{\rho_{\lambda}^{\text{der}}\}$ has multiplicity more than one. This contradicts with \autoref{table1}.
	\end{proof}
	
	We separate the following three cases:
	\begin{enumerate}[(a)]
		\item There exist infinite many $\lambda_1$ such that each $\rho_{\lambda_1}$ contains a 2-dimensional component.
		\item There exist infinite many $\lambda_1$ such that the decomposition of $\rho_{\lambda_1}$ has dimensional type $4+4$ with each component essentially self-dual and odd.
		\item There exist infinite many $\lambda_1$ such that (\ref{1}) hold.
	\end{enumerate}
	
	In case (a), by \autoref{Boron}.(i) for some $\lambda_1$ this 2-dimensional component $\varphi_{\lambda_1}$ fits into a strictly compatible system $\{\varphi_{\lambda}\}$. Then consider compatible system $\{\rho_{\lambda}\oplus\varphi_{\lambda}\}$. The semisimple rank at place $\lambda_1$ is the same as that of $\rho_{\lambda_0}$. However Goursat's lemma guarantees the semisimple rank at place $\lambda_0$ is strictly large by 1.
	
	In case (b), as the formal characters of $\rho_{\lambda_0}^{\text{der}}$ in the cases we consider have no repeated weights, the two components are both Lie-irreducible. We show for some $\lambda_1$, one of the component $\varphi_{\lambda_1}$ fits into a strictly compatible system $\{\varphi_{\lambda}\}$. Then the semisimple rank of compatible system $\{\rho_{\lambda}\oplus\varphi_{\lambda}\}$ at places $\lambda_0$ and $\lambda_1$ do not match due to Goursat's lemma again.
	
	To do so again we check \autoref{Lithium}. Only condition \autoref{Lithium}.(iv) requires explanation. If the Lie type of any component is of type A, then \autoref{Beryllium}.(iii) gives the conclusion. The only remaining case is both components have Lie type $(\text{Sp}_4,\text{Std})$ for sufficiently large $\lambda_1$. We first show they are residually irreducible for sufficiently large $\lambda_1$. If there are infinitely many $\lambda_1$ such that one of the component $\varphi_{\lambda_1}$ is residually reducible. By \autoref{Beryllium}.(i), one must have
	$$\overline{\varphi}_{\lambda_1}^{\text{ss}}=\overline{\sigma}_{\lambda_1,1}\oplus\overline{\sigma}_{\lambda_1,2}$$
	where $\overline{\sigma}_{\lambda_1,i}$ are 2-dimensional irreducible representations. As $\varphi_{\lambda_1}$ has an odd similitude character $\chi_{\lambda_1}$. Either one has
	$$\overline{\sigma}_{\lambda_1,i}^{\vee}\cong\overline{\sigma}_{\lambda_1,i}\otimes\overline{\chi}_{\lambda_1},i=1,2$$
	in which case both $\overline{\sigma}_{\lambda_1,i}$ are odd, or one has
	$$\overline{\sigma}_{\lambda_1,1}^{\vee}\cong\overline{\sigma}_{\lambda_1,2}\otimes\overline{\chi}_{\lambda_1},\overline{\sigma}_{\lambda_1,2}^{\vee}\cong\overline{\sigma}_{\lambda_1,1}\otimes\overline{\chi}_{\lambda_1}$$
	in which case the derived subgroup $\underline{G}_{\lambda_1}^{\mathrm{der}}$ of algebraic envelope at $\lambda_1$ would be $\text{SL}_2$ and this contradicts with \autoref{Beryllium}.(i). We denote by $\overline{\sigma}_{\lambda_1}$ any odd component of $\overline{\varphi}_{\lambda_1}^{\text{ss}}$. Then by \autoref{Carbon}, after possibly enlarging $E$, there exists a 2-dimensional strictly compatible system $\{\sigma_{\lambda}\}$ such that the semisimple reduction of $\sigma_{\lambda_1}$ is $\overline{\sigma}_{\lambda_1}$ for sufficiently large $\lambda_1$. But consider compatible system $\{\alpha_{\lambda}=\rho_{\lambda}\oplus\sigma_{\lambda}\}$. Denote by $s$ the semisimple rank of algebraic envelope at place $\lambda_1$. We have $s=3$ in case (10) and $s=4$ in cases (12), (13) and (14). Then by \autoref{Beryllium}.(i) the semisimple rank of $\alpha_{\lambda_1}$ is $s$. However Goursat's lemma asserts the semisimple rank of $\alpha_{\lambda_0}$ is $s+1$. This contradicts with \autoref{Hydrogen}.(ii).
	
	Now if one of the component in (b) is residually irreducible after restricting to $\text{Gal}_{\mathbb{Q}(\zeta)_{\ell_1}}$ for $\lambda_1|\ell_1$, we are done. Hence we assume both components, which we denote by $\varphi_{\lambda_1}$ and $\varphi_{\lambda_1}'$, are residually reducible after restricting to $\mathbb{Q}(\zeta_{\ell_1})$ for $\lambda_1$ sufficiently large. Then their semisimple reductions are induced by two-dimensional Lie-irreducible representations over quadratic fields $K_{\lambda_1}$ and $K_{\lambda_1}'$ inside number field $L$ in \autoref{Beryllium}.(ii). 
	$$\overline{\varphi}_{\lambda_1}=\text{Ind}^{\mathbb{Q}}_{K_{\lambda_1}}\overline{\sigma},\ \overline{\varphi}_{\lambda_1}'=\text{Ind}^{\mathbb{Q}}_{K_{\lambda_1}'}\overline{\sigma}'$$
	Hence there exists infinitely many $\lambda_1$ such that $K_{\lambda_1}$ coincide and $K_{\lambda_1}'$ coincide. However, then the density of trace zero primes under $\rho_{\lambda_1}$ is not zero. This contradicts with the fact $\rho_{\lambda_0}$ is Lie-irreducible.
	
	Finally in case (c). As the semisimple rank is 3, the only possible Lie type of $\rho_{\lambda_0}$ would be case (10), i.e. $\mathrm{SO}_7$ with spin representation. Our proof needs the extra Hodge-Tate condition in the statement in \autoref{main}. We write irreducible decomposition $\rho_{\lambda_1}=\sigma_{\lambda_1,1}\oplus\sigma_{\lambda_1,2}$. Now consider the compatible system $\{\varphi_{\lambda}=\rho_{\lambda}\otimes\rho_{\lambda}\}$. At place $\lambda_0$ the irreducible decomposition is:
	$$\rho_{\lambda_0}\otimes\rho_{\lambda_0}=\sigma_0\oplus\sigma_1\oplus\sigma_2\oplus\sigma_3$$
	where the Lie type of $\sigma_i$ is $\wedge^i(\mathrm{SO}_7,\mathrm{Std})$. At place $\lambda_1$ the irreducible decomposition is (for simplicity we omit index $\lambda_1$ in $\sigma_{\lambda_1,i}$):
	\begin{align*}
		\rho_{\lambda_1}\otimes\rho_{\lambda_1}&=\left(\sigma_1\otimes\sigma_1\right)\oplus\left(\sigma_2\otimes\sigma_2\right)\oplus\left(\sigma_1\otimes\sigma_2\right)^2\\&=\left(\mathrm{Sym}^2(\sigma_1)\oplus\wedge^2(\sigma_1)\right)\oplus\left(\mathrm{Sym}^2(\sigma_2)\oplus\wedge^2(\sigma_2)\right)\oplus\left(\chi\oplus\tau\right)^2
	\end{align*}
	where $\chi$ is a character and $\tau$ is a 15-dimensional irreducible representation. Now consider its subrepresentation:
	$$\alpha=\mathrm{Sym}^2(\sigma_1)\oplus\mathrm{Sym}^2(\sigma_2)$$
	The Lie type of each component is $\mathrm{SO}_6$. Moreover one has
	$$\mathrm{Sym}^2\left(\mathrm{SL}_4,\mathrm{Std}\right)\oplus\mathrm{Sym}^2\left(\mathrm{SL}_4,\mathrm{Std}^{\vee}\right)\cong\wedge^3\left(\mathrm{SO}_6,\mathrm{Std}\right)$$
	Hence the derived subgroup $\textbf{G}_{\alpha}^{\text{der}}$ of $\alpha$ is $\text{SO}_6/\{\pm E_6\}$. One twists a compatible system of (Hodge-Tate) characters to $\{\varphi_{\lambda}\}$ so that the algebraic monodromy group $\textbf{G}_{\alpha}=\mathbb{G}_m\textbf{G}_{\alpha}^{\text{der}}$. One has the following isomorphism:
	$$\pi:\text{GO}_6\to\textbf{G}_{\alpha}$$
	$$g\mapsto\det(g)^{-1/3}\wedge^3(g)$$
	Hence one writes $\wedge^3\psi=\beta\otimes\alpha$ where $\beta$ is a (Hodge-Tate) character and $\psi:\text{Gal}_{\mathbb{Q}}\to\text{GO}_6(\overline{E}_{\lambda_1})$ is unramified almost everywhere and its restriction to $\text{Gal}_{\mathbb{Q}_{\ell}}$ is crystalline except for a finite set of $\lambda_1$, where $\lambda_1\mid\ell$.
	
	We show for suitable $\lambda_1$ this $\psi$ fits into a compatible system. To do so we check conditions in \autoref{Lithium}. (i) is obvious. To check (ii), as $\psi$ is essentially self-dual, we denote by $\chi$ its similitude character. Then $\alpha^{\vee}\cong\alpha\otimes\chi^3$. We want to show $\chi(c)=1$ for some complex conjugation $c$. Suppose otherwise $\chi(c)=-1$. Since $\rho_{\lambda_1}\otimes\rho_{\lambda_1}$ has a similitude character $\eta$ with $\eta(c)=1$, and there is no other 10-dimensional components than $\text{Sym}^2(\sigma_1)$ and $\text{Sym}^2(\sigma_2)$ in the decomposition, one must have $\alpha^{\vee}\cong\alpha\otimes\eta$. Hence one has $\alpha\cong\alpha\otimes\chi^3\eta^{-1}$. As $\chi^3(c)\eta^{-1}(c)=-1$, this shows the set of eigenvalues of $\alpha(c)$ is symmetric under multiplying $-1$. Hence there are exactly ten eigenvalues $-1$ and ten eigenvalues $1$ of $\alpha(c)$. However one has the list \autoref{table2}. As one cannot choose $\sigma_1$ and $\sigma_2$ in the list such that the union of eigenvalues of $\text{Sym}^2(\sigma_1)$ and $\text{Sym}^2(\sigma_2)$ satisfies this, one must have $\chi(c)=1$. Hence $\psi$ is odd.
	\begin{table}
		\caption{}
		\label{table2}
		\centering
		\begin{tabular}{|c|c|c|}
			\hline
			eigenvalues of $c$ on 4-dim $\sigma$ & no. of eigenvalues $-1$ in $\mathrm{Sym}^2(\sigma)$ & no. of eigenvalues $1$ in $\mathrm{Sym}^2(\sigma)$\\
			\hline
			$\{1,1,1,1\}$ & 0 & 10\\
			\hline
			$\{1,1,1,-1\}$ &3&7\\
			\hline
			$\{1,1,-1,-1\}$&4&6\\
			\hline
			$\{1,-1,-1,-1\}$&3&7\\
			\hline
			$\{-1,-1,-1,-1\}$&0&10\\
			\hline
		\end{tabular}
	\end{table}
	
	To check (iii), by \autoref{Sodium}, after enlarging $\lambda_1$, it is enough to show $\psi$ is regular under our assumption on Hodge-Tate weights. Denote by $\text{Ht}(\sigma_1)=\{a_1,a_2,a_3,a_4\}$ the set of Hodge-Tate weights of $\sigma_1$, then $\text{Ht}(\sigma_2)=\{n-a_1.n-a_2.n-a_3,n-a_4\}$ for some integer $n$. We know $\{a_1,a_2,a_3,a_4,n-a_1,n-a_2,n-a_3,n-a_4\}$ are distinct. Denote by $\text{Ht}(\psi)=\{x_1,x_2,x_3,x_4,x_5,x_6\}$. Then $\text{Ht}(\beta\otimes\alpha)$ is the multiset:
	$$A=\{a_i+a_j+m,2n-(a_i+a_j)+m,1\le i\le j\le4\}$$
	And $\text{Ht}(\wedge^3(\psi))$ is the multiset
	$$B=\{x_i+x_j+x_k,1\le i<j<k\le6\}$$
	We have $A=B$. Consider following condition.
	
	(P): $\{a_1,a_2,a_3,a_4,n-a_1,n-a_2,n-a_3,n-a_4\}$ are distinct and there exist no three distinct elements of them form a 3-term arithmetic progression.
	\begin{lem}
		Under condition (P), we have $\{x_1,x_2,\cdots,x_6\}$ are distinct.
	\end{lem}
	\begin{proof}
		Since we only care about distinctness of $\{x_1,x_2,\cdots,x_6\}$, we can replace each $x_i$ with $x_i-m/3$, each $a_i$ with $a_i-m/2$ and $n$ with $n-m$. This does not change the fact $\{a_1,a_2,a_3,a_4,n-a_1,n-a_2,n-a_3,n-a_4\}$ satisfies condition (P). Hence we can assume $m=0$. Similiarly, replace $a_i$ with $a_i+n$ and $x_i$ with $x_i+2n/3$ we can assume $n=0$.
		
		If say $x_5=x_6$, we can find six pairs of elements in $B$ that coincide:
		$$x_5+x_i+x_j=x_6+x_i+x_j,1\le i<j\le 4$$
		Write $B_2$ the multiset consist of them and $B_1$ to be its compliment multiset in $B$. We also divide $A$ into the disjoint union of two multisets $A_1$ and $A_2$.
		$$A_1=\{2a_1,2a_2,2a_3,2a_4,-2a_1,-2a_2,-2a_3,-2a_4\}$$
		$$A_2=\{\pm(a_i+a_j),1\le i<j\le 4\}$$
		Due to condition (P), we see elements in $A_1$ are distinct with any element in $A$. Hence we have $A_2=B_2$ consist of six pairs of elements with equal values. Suppose one has indexes $i\not=j$ and $s\not=t$ such that $(a_i+a_j)=\varepsilon(a_s+a_t)$ where $\varepsilon=\pm1$. Then $\{i,j\}$ and $\{s,t\}$ must be disjoint otherwise would contradict with condition (P). Now we write $\{s,t\}$ to be the compliment set of $\{i,j\}$ in $\{1,2,3,4\}$. Then we have six equations:
		$$(a_i+a_j)=\varepsilon_{i,j}(a_s+a_t),1\le i<j\le4$$
		where $\varepsilon_{i,j}=\pm1$. Then one must have all $\varepsilon_{i,j}=-1$ since other situations all contradict with (P).
		\begin{enumerate}[(i)]
			\item If all $\varepsilon_{i,j}=1$ then $a_1=a_2=a_3=a_4$.
			\item If some $\varepsilon=-1$ which implies $a_1+a_2+a_3+a_4=0$, and some $\varepsilon_{i,j}=1$, then $a_i=-a_j$.
		\end{enumerate}
		Hence the only possible partition of $A_2=B_2$ is:
		$$(a_i+a_j)=-(a_s+a_t),\forall 1\le i<j\le4$$
		with $\{s,t\}$ being the compliment set of $\{i,j\}$ in $\{1,2,3,4\}$.
		
		Next we claim we can rearrange the order of $\{a_1,a_2,a_3,a_4\}$ so that one has
		\begin{equation}\label{3}
		x_5+x_i+x_j=a_i+a_j,\forall1\le i<j\le4
		\end{equation}
		To do so we define $S$ to be the set of all 2-element subset of $\{1,2,3,4\}$. Then the corresponding from $A_2$ to $B_2$ gives a bijection $f$ on $S$. Namely if $f(\{i,j\})=\{s,t\}$ then $x_5+x_i+x_j=a_s+a_t$. The claim is equivalent to show if $p,q\in S$ have intersection then $f(p)$ and $f(q)$ have intersection. Say on the contrary we have $\{i,j,k,l\}=\{s,t,u,v\}=\{1,2,3,4\}$ and
		$$x_5+x_i+x_j=a_s+a_t$$
		$$x_5+x_i+x_k=a_u+a_v$$
		Then take summation of the two equations gives $x_i=x_l$. However this would produce more pairs with equal values in $B$, say $x_i+x_j+x_k=x_l+x_j+x_k$. But elements in multiset $A_1$ are distinct hence a contradiction. 
		
		By (\ref{3}) one gets $a_i=x_i+x_5/2$ for $1\le i\le4$. Write $u=3x_5/2$. We divide $B_1$ into four pairs such that in each pair the summation of two elements is zero.
		\begin{equation}\label{4}
		A_{1,i}=\left\{\left(\sum_{k\not=i}a_k-u\right),\left(a_i+u\right)\right\},1\le i\le4
		\end{equation}
		
		As $A_1=B_1$, we have a bijection $g$ on $\{1,2,3,4\}$ so that $2a_i\in A_{1,g(i)}$. This gives four equations when $1\le i\le4$. We say the equation given by index $i$ is of type 1 or 2, if $2a_i$ is the first or second term in (\ref{4}) for $A_{1,g(i)}$. Denote by $r$ the number of type 1 equations. 
		
		We first rule out $r\ge3$. Say we have three distinct indexes $i,j,k$ whose corresponding equations are of type 1. Then $\{i,j,k\}$ and $\{g(i),g(j),g(k)\}$ have intersection, say $g(i)=j$. Then
		$$\sum_{1\le l\le4}a_l-u=2a_i+a_j$$
		$$\sum_{1\le l\le4}a_l-u=2a_j+a_{g(j)}$$
		Hence $2a_i=a_j+a_{g(j)}$, this contradicts with condition (P) no matter which index $g(j)$ is.
		
		Now for $r\le2$. We show the coefficient matrix $M$ of the system of linear equations given by the four index $\{1,2,3,4\}$ is invertible. Hence the obvious solution $a_1=a_2=a_3=a_4=u$ is the unique solution and this contradicts with condition (P) again.
		
		Let $U=\diag\{d_1,d_2,d_3,d_4\}$ be a diagonal matrix with $d_i=\pm1$ and $\omega=\left(w_1,w_2,w_3,w_4\right)^T$ be a vector such that if $d_i=1$ then $w_i=1$ and if $d_i=-1$ then $w_i=0$. Denote by $\textbf{1}$ the vector $\left(1,1,1,1\right)^T$, by $I_4$ the identity matrix. Denote by $P$ the permutation matrix corresponding to the action of $g$ on $\{1,2,3,4\}$. Then one has
		$$M=2I_4+UP-\textbf{1}\omega^T$$
		The $i$-th row of $M$ corresponds to type 1 equation if $d_i=1$ and type 2 equation if $d_i=-1$. And $r=\omega^T\textbf{1}$. Denote by $K=2I_4+UP$. As $UP$ is an orthogonal matrix, $K$ has no zero eigenvalues. Hence $K$ is invertible. Then one has
		$$\det(M)=(1-\omega^TK^{-1}\textbf{1})\det(K)$$
		Hence we only need to show $h=\omega^TK^{-1}\textbf{1}\not=1$ under the condition $r\le2$. As $\Vert UP/2\Vert=1/2<1$, one has the following power series expansion.
		\begin{align*}
			h&=\frac{1}{2}\omega^T\left(I_4+\frac{UP}{2}\right)^{-1}\textbf{1}\\&=\frac{1}{2}\sum_{m\ge0}\left(-\frac{1}{2}\right)^m\omega^T\left(UP\right)^m\textbf{1}
		\end{align*}
		We write $a_m=\omega^T\left(UP\right)^m\textbf{1}$. As $UP$ is a permutation matrix with possible negative signs on the entries, we have a smallest postive integer $N$ such that $(UP)^N=\varepsilon I_4$, where $\varepsilon=\pm1$. Then $a_{m+N}=\varepsilon a_m$. Hence we have
		$$h=\frac{\sum_{m=0}^{N-1}\left(-\frac{1}{2}\right)^ma_m}{2-2\varepsilon\left(-\frac{1}{2}\right)^N}$$
		If $h=1$, one has:
		\begin{equation}\label{5}
			\sum_{m=0}^{N-1}(-1)^m2^{N-1-m}a_m=2^N-(-1)^N\varepsilon2
		\end{equation}
		Here $a_0=\sum_{1\le i\le4}w_i=r$, $a_1=\sum_{1\le i\le4}w_id_i=r$ and in general $|a_m|\le r$. One has following estimate on the two sides.
		$$\text{LHS}\le2^{N-2}r+r\left(1+2+\cdots+2^{N-3}\right)\le2^N-2\le\text{RHS}$$
		Hence for the equation holds, one must have $r=2$ and $a_2=2$, $a_3=-2$. Say $w_i=w_j=1$ and $w_s=w_t=0$ where $\{i,j,s,t\}=\{1,2,3,4\}$. Then one has
		$$a_2=\sum_{1\le l\le4}w_ld_ld_{g(l)}=2$$
		This means $g(\{i,j\})=\{i,j\}$. However, this would implies $a_3=\sum_{1\le l\le4}w_ld_ld_{g(l)}d_{g^2(i)}=2$ again. This is a contradiction.
	\end{proof}
	
	We check \autoref{Lithium}.(iv). If the semisimple reduction of $\psi$ is reducible after restricting to $\text{Gal}_{\mathbb{Q}(\zeta_{\ell})}$, one of the semisimple reductions of $\text{Sym}^2(\sigma_1)$ and $\text{Sym}^2(\sigma_2)$ would also be reducible after restricting to $\text{Gal}_{\mathbb{Q}(\zeta_{\ell})}$. However those symmetric squares are Lie-irreducible and of type A, hence by \autoref{Beryllium}.(iii) this cannot happen.
	
	Finally, as $\psi$ fits into a compatible system for some $\lambda_1$, the representation $\alpha_{\lambda_1}=\alpha=\beta^{-1}\wedge^3\psi$ also fits into a compatible system $\{\alpha_{\lambda}\}$. Consider compatible system $\{\varphi_{\lambda}\oplus\alpha_{\lambda}\}$. At place $\lambda_1$ the semisimple rank is the same as that of $\{\varphi_{\lambda}\}$. However, at place $\lambda_0$, since the derived subgroup of algebraic monodromy groups of $\sigma_i$ for $i\ge1$ is some quotient of $\text{SO}_7$, by Goursat's lemma, the semisimple rank of $\varphi_{\lambda_0}\oplus\alpha_{\lambda_0}$ is strictly larger than that of $\varphi_{\lambda_1}\oplus\alpha_{\lambda_1}$. This contradicts with \autoref{Hydrogen}.(ii).
	
	\subsection{Cases (6), (8), (9)}
	
	The three cases we consider are rectangular representations, hence $\textbf{G}_{\lambda_1}^{\mathrm{der}}$ can only be one of the six cases in \autoref{Chromium}.(ii) and (iii) We may assume \autoref{Chromium}.(iii).(e) does not happen for sufficiently large $\lambda_1$ since this case has been taken care of before. We can also rule out \autoref{Chromium}.(iii).(d) since in such case $\rho_{\lambda_1}$ would be irreducible. Hence we may assume each $\rho_{\lambda_1}$ is of type A.
	
	Due to same reason as in the cases (10), (12), (13), (14), $\rho_{\lambda_1}$ cannot contain 1 or 3 or 5-dimensional components. By \autoref{Titanium}, we separated following two cases.
	\begin{enumerate}[(a)]
		\item There are infinitely many $\lambda_1$ such that the decompositions of $\rho_{\lambda_1}$ have dimensional type $2+2+2+2$, $6+2$, $4+2+2$ or $4+4$ such that the two 4-dimensional components are essentially self-dual and odd.
		\item There are infinitely many $\lambda_1$ such that (\ref{1}) is true.
	\end{enumerate}
	
	In case (a), we show for some $\lambda_1$, each component of $\rho_{\lambda_1}$ fits into a strictly compatible system, then this contradicts with the irreducibility of $\rho_{\lambda_0}$. For sufficiently large $\lambda_1$, the 2-dimensional component fits into a strictly compatible system due to \autoref{Boron}.(i). For other components we apply \autoref{Lithium}. Conditions \autoref{Lithium}.(i) and \autoref{Lithium}.(iii) are obvious. When the dimensional type is $6+2$ or $4+2+2$, the 4 or 6-dimensional component is essentially self-dual and odd due to \autoref{Potassium}. Hence condition \autoref{Lithium}.(ii) holds. Finally, since the formal character of $\rho_{\lambda_0}^{\text{der}}$ has no repeated weights in the cases we consider, the 4 or 6-dimensional component is Lie-irreducible. Also as explained at the beginning, these $\rho_{\lambda_1}$, hence the irreducible components, are of type A. Hence by \autoref{Beryllium}.(iii), condition \autoref{Lithium}.(iv) holds.
	
    In case (b). One writes $\rho_{\lambda_0}=f\otimes g$ where $f$ is a 2-dimensional irreducible representation. Then there exists a 3-dimensional irreducible subrepresentation $\varphi_{\lambda_0}$ of $\rho_{\lambda_0}\otimes\rho_{\lambda_0}^{\vee}$ such that the restriction of $\varphi_{\lambda_0}$ to $\text{Gal}_{\mathbb{Q}_{\ell}}$ is the trace zero subrepresentation of $f\otimes f^{\vee}$. Since $\rho_{\lambda_0}$ is regular, so is $f$. Hence $\varphi_{\lambda_0}$ is regular. The Lie type of $\varphi_{\lambda_0}$ is $\mathrm{SO}_3$. Recall our $\lambda_0$ runs through an infinite set $\mathcal{L}$. We choose $\lambda_0$ large enough so that \autoref{Boron}.(ii) is satisfied for compatible system $\{\rho_{\lambda}\otimes\rho_{\lambda}^{\vee}\}$. Hence it extends to a compatible system $\{\varphi_{\lambda}\}$.
	
	Consider the 11-dimensional strictly compatible system $\{\rho_{\lambda}\oplus\varphi_{\lambda}\}$. At place $\lambda_0$ the semisimple rank is the same as that of $\rho_{\lambda_0}$. However at place $\lambda_1$, by Goursat's lemma the derived subgroup of algebraic monodromy group has strictly larger rank. This is a contradiction.
	
	\section*{Acknowledgement}


\end{document}